\documentclass[11pt,a4paper,reqno]{amsart}

\usepackage{amsmath,amsfonts,amssymb,amsthm}

\usepackage{bm}

\usepackage{verbatim}

\usepackage[all]{xy}

\usepackage{hyperref}

\newlength{\wideitemsep}
\let\olditem\item
\renewcommand{\item}{\setlength{\itemsep}{\wideitemsep}\olditem}

\numberwithin{equation}{section}

\renewcommand{\geq}{\geqslant}
\renewcommand{\leq}{\leqslant}

\newcommand{\lra}{\longrightarrow}

\makeatletter
\newtheorem*{rep@theorem}{\rep@title}
\newcommand{\newreptheorem}[2]{%
\newenvironment{rep#1}[1]{%
 \def\rep@title{#2 \ref{##1}}%
 \begin{rep@theorem}}%
 {\end{rep@theorem}}}
\makeatother

\newtheorem{theorem}{Theorem}[section]
\newtheorem{lemma}[theorem]{Lemma}
\newtheorem{proposition}[theorem]{Proposition}

\newreptheorem{theorem}{Theorem}
\newreptheorem{corollary}{Corollary}

\theoremstyle{definition}

\theoremstyle{remark}
\newtheorem{remark}[theorem]{Remark}

\newcommand{\cA}{\mathcal{A}}

\newcommand{\cE}{\mathcal{E}}

\newcommand{\cI}{\mathcal{I}}
\newcommand{\cO}{\mathcal{O}}
\newcommand{\cP}{\mathcal{P}}
\newcommand{\cL}{\mathcal{L}}
\newcommand{\cM}{\mathcal{M}}

\newcommand{\XX}{X \times_S X}
\newcommand{\wf}{Weierstra{\ss} fibration\ }

\DeclareMathOperator{\coh}{Coh}

\DeclareMathOperator{\pic}{Pic}

\DeclareMathOperator{\Hom}{Hom}

\DeclareMathOperator{\ch}{ch}

\DeclareMathOperator{\id}{id}

\DeclareMathOperator{\Ext}{Ext}

\DeclareMathOperator{\rank}{rank}

\DeclareMathOperator{\shom}{\mathcal{H}\!\textit{om}\,}

\DeclareMathOperator{\sext}{\mathcal{E}\!\textit{xt}\,}

\DeclareMathOperator{\codim}{codim}

\newcommand{\wh}{\widehat}
\newcommand{\Coh}{\mathrm{Coh}}
\newcommand{\Lm}{{L_m}}

\begin{document}

\title{Duality Spectral Sequences for Weierstrass Fibrations and Applications}

\author[Jason Lo]{Jason Lo}
\address{Department of Mathematics, California State University Northridge, 18111 Nordhoff Street, Northridge CA 91330, USA}
\email{jason.lo@csun.edu}
\urladdr{http://sites.google.com/site/chiehcjlo}

\author{Ziyu Zhang}
\address{Institute of Algebraic Geometry, Leibniz University Hannover, Welfengarten 1, 30167 Hannover, Germany}
\email{zhangzy@math.uni-hannover.de}
\urladdr{http://ziyuzhang.github.io}

\keywords{Weierstra{\ss} fibration, duality spectral sequence, Fourier-Mukai tranform, stability}
\subjclass[2010]{Primary 14D20; Secondary: 14J30, 14J32, 14J60}

\begin{abstract}
We study  duality spectral sequences for Weierstra{\ss} fibrations. Using these  spectral sequences, we show that on a $K$-trivial Weierstra{\ss} threefold over a $K$-numerically trivial surface, any line bundle of nonzero fiber degree is taken by a Fourier-Mukai transform to a slope stable locally free sheaf.
\end{abstract}

\maketitle

\section{Introduction}

Given two smooth projective varieties $X$ and $Y$ related by a Fourier-Mukai transform $\Phi : D^b(X) \to D^b(Y)$, it is usually not the case that a slope stable coherent sheaf $E$ on $X$ is taken by $\Phi$ to a slope stable coherent sheaf on $Y$.  It is therefore natural to ask: under what circumstances does this happen? 

Yoshioka gave examples and counterexamples to the above question on Abelian and K3 surfaces \cite{YosNote}.  On threefolds, Bridgeland-Maciocia showed in \cite[Theorem 1.4]{BMef} that, if $X$ and $Y$ are dual elliptic threefolds with a corresponding Fourier-Mukai transform $\Phi : D^b(X) \to D^b(Y)$, then for any rank-one torsion-free sheaf $E$ on $X$, there is a fixed line bundle $L$ (depending only on $\ch_1(E)$) such that $E \otimes L$ is taken by $\Phi$ to a torsion-free sheaf that is Gieseker stable with respect to a polarisation $\omega$ on $Y$, where $\omega$ depends on $\ch(E)$.

The main result of this paper is as follows:

\textbf{Theorem \ref{theorem:main1}}.
\textit{Suppose $p : X \to S$ is a  Weierstra{\ss} threefold where $X$ is $K$-trivial and $K_S$ is numerically trivial.    Then for any ample class $\omega$ on $X$ and any line bundle $M$ on $X$ of nonzero fiber degree, the Fourier-Mukai transform of $M$ (up to a shift) is a $\mu_\omega$-stable locally free sheaf.}

A main ingredient in the proof of Theorem \ref{theorem:main1} is a pair of spectral sequences on the interplay between a Fourier-Mukai transform and the derived dual functor on elliptic threefolds, discussed in Section \ref{sec:specseq}.  These spectral sequences are   inspired by analogues on Abelian threefolds due to Maciocia-Piyaratne \cite[Proposition 4.2]{MacPir1}.  

We  note that some of the techniques in this paper have also appeared in the work of Oberdieck-Shen \cite{OS1}.  In their work,  they study the transform of stable pairs (in the sense of Pandharipande-Thomas \cite{PT}) under an autoequivalence of the derived category of an elliptic threefold, thereby giving a partial proof to the modulairty conjecture on PT invariants due to Huang-Katz-Klemm.

The paper is organized as follows: in Section \ref{sec2} we recall the Fourier-Mukai functor for Weierstra{\ss} fibrations and the dualizing functor, as well as Theorem \ref{thm:commutativity} that establishes the commutativity of the two functors. In Section \ref{sec:specseq} we apply Theorem \ref{thm:commutativity} and the technique of spectral sequences to study the behavior of sheaves under Fourier-Mukai transforms.  In Section \ref{sec4}, we give a criterion for reflexive sheaves to be taken to locally free sheaves by a Fourier-Mukai transform, thus paving the way for the proof of  the main result of this paper,   Theorem \ref{theorem:main1}.

\textbf{Acknowledgements.} We would like to thank Arend Bayer, Antony Maciocia and Dulip Piyaratne for valuable discussions that led to the conception of this project.  The first author would also like to thank the Institute of Algebraic Geometry at Leibniz University  Hannover for their hospitality and support during the final stage of this project in May-June 2017. The second author would like to thank Alastair Craw for his support via EPSRC grant EP/J019410/1 and the support from Leibniz University Hannover during various stages of the work.

\section{Commutativity of Fourier-Mukai and dualizing functors}\label{sec2}

In this section, all functors are derived unless otherwise specified.

\subsection{The Fourier-Mukai functor}

We fix some notations first. Let
\begin{equation}
\label{eqn:fibxs}
p: X \lra S
\end{equation}
be a \wf in the sense of \cite[Definitions 6.8, 6.10]{FMNT}. The fibration has a section $\sigma: S \to X$ whose image lies in the smooth locus of the morphism $p$. We write $\Theta = \sigma(S)$ and $h = p|_{\Theta}: \Theta \to S$. Then $h$ is clearly an isomorphism.

In the case of our interest, we further assume that both $X$ and $S$ are smooth. We label the relevant morphisms as in the following fiber diagram

\begin{equation}
\xymatrix{
\XX \ar[r]^-{\pi_2} \ar[d]_-{\pi_1} \ar[rd]^-{\pi} & X \ar[d]^p \\
X \ar[r]_p & S.
}
\end{equation}

We define an integral functor whose kernel is given by the sheaf
\begin{equation}
\label{eqn:FM-kernel}
\cP = \cI_{\Delta} \otimes \pi_1^*\cO_X(\Theta) \otimes \pi_2^*\cO_X(\Theta) \otimes \pi^*\omega^{-1}
\end{equation}
where $\cI_{\Delta}$ is the ideal sheaf of the image $\Delta$ of the diagonal morphism $\delta: X \hookrightarrow \XX$, and
$$ \omega = R^1p_*\cO_X $$
as introduced in \cite[p.191, l.2]{FMNT}.

Then we have the following result:

\begin{theorem}
\label{thm:FM-formula}
The integral functor $\Phi: D^b(X) \to D^b(X)$ given by
\begin{equation}
\label{eqn:FM-formula}
\Phi(\cE) := \pi_{2*}(\pi_1^* \cE \otimes \cP)
\end{equation}
is well-defined and an equivalence of categories.
\end{theorem}

\begin{proof}
This is \cite[Theorem 6.18]{FMNT} or \cite[Theorem 2.12]{BK06}.
\end{proof}

\begin{remark}
We point out that the Fourier-Mukai kernel \eqref{eqn:FM-kernel} is the one defined in \cite[Definition 6.14]{FMNT}. A slightly different Fourier-Mukai kernel, with the last factor $\pi^*\omega^{-1}$ in \eqref{eqn:FM-kernel} being omitted, was used in \cite{BK06}. Since they only differ by a line bundle pulled back from the base $S$, most of our discussion is valid regardless of which kernel we choose. We will be precise when the choice of the kernel does affect the computation. \qed
\end{remark}

The Fourier-Mukai transform $\Phi$ defined in the above theorem will be one of the main functors that we will consider. Now we look at the other functor that will play a key role.

\subsection{The dualizing functor}

Let $Z$ be a Gorenstein variety and $\cE \in D^b(Z)$. We define
$$ \Delta_Z(\cE) := \shom_{\cO_Z}(\cE, \cO_Z). $$
It is familiar that when $Z$ is smooth, we have
\begin{equation}
\label{eqn:dual-bounded}
\Delta_Z(\cE) \in D^b(Z)
\end{equation}
and
\begin{equation}
\label{eqn:dual-dual}
(\Delta_Z \circ \Delta_Z)(\cE)=\cE.
\end{equation}
In fact, both relations still hold when $Z$ is Gorenstein; see for example \cite[Example 3.20]{Nee10}.

We will use the dualizing functor on $X$ and $\XX$. Notice that $\XX$ is in general singular, but still Gorenstein. Indeed, the projection morphism $\pi_1: \XX \to X$ is a base change of the morphism $p: X \to S$. Since $p: X \to S$ is flat with Gorenstein fibers, the same is true for $\pi_1: \XX \to X$, hence they are both Gorenstein morphisms; see \cite[p.349, l.5]{FMNT}. It follows that $\XX$ is Gorenstein by transitivity \cite[Proposition C.1.1]{FMNT}. Hence both \eqref{eqn:dual-bounded} and \eqref{eqn:dual-dual} hold for $X$ and $\XX$.

\subsection{The commutativity of functors}

Now we discuss the commutativity of the Fourier-Mukai functor and the dualizing functors. Recall that for the Weierstra{\ss} fibration \eqref{eqn:fibxs}, an $S$-involution is a morphism $\iota: X \to X$ satisfying $\iota \circ \iota = \id_X$ and $p \circ \iota = p$. The key result concerning the commutativity of functors is the following:

\begin{theorem}
\label{thm:commutativity}
There exists a line bundle $L \in \pic(S)$ and an $S$-involution $\iota: X \to X$, such that for any $\cE \in D^b(X)$, we have
\begin{equation}
\label{eqn:2comp}
(\Delta_X \circ \Phi) (\cE) = \iota^* (\Phi \circ \Delta_X) (\cE) \otimes p^*L [1].
\end{equation}
\end{theorem}

\begin{proof}
This is \cite[Corollary 3.3]{BK06}. Indeed, we only need to observe that our dualizing functor $\Delta_X$ is a special case of the dualizing functor there by setting $\cL = \cO_S$ and denote $\cM^\vee \otimes \cA$ by $L$ in the formulation of \cite[Corollary 3.3]{BK06}. From earlier discussion in \cite{BK06} we know that $L$ is in fact a line bundle.
\end{proof}

This is a very convenient tool for our discussion in next section. As pointed out in \cite[p.301, l.25]{BK06}, this result generalize the classical result of Mukai in \cite[(3.8)]{Muk81}. Before discussing its applications we make some remarks concerning $L$ and $\iota$ in the statement of the theorem.

\begin{remark}
For our purpose it will be not at all relevant to know what $L$ precisely is. However it can be computed explicitly, which depends on the choice of the Fourier-Mukai kernel. We explain it very briefly using the kernel $\cP$ defined in \eqref{eqn:FM-kernel}. We know that $L = \cM^\vee \otimes \cA$ in the notation of \cite[Corollary 3.3]{BK06}. We just need to compute $\cM$ and $\cA$.

The definition of $\cA$ is given in \cite[p.300, l.14]{BK06} by
$$ \omega_{X/B} = p^* \cA. $$
We use the relation \cite[(6.5)]{FMNT} to conclude that $\cA = \omega^{-1}$.

The definition of $\cM$ is given in \cite[Proposition 2.3]{BK06}. To compute $\cM$ we consider the inclusion
$$\sigma \times \sigma: \Theta \times_S \Theta \hookrightarrow \XX.$$
Under the canonical isomorphism $\Theta \times_S \Theta \cong \Theta$ we know that $\pi \circ (\sigma\times\sigma)=p \circ \sigma = h$ which is an isomorphism from $\Theta$ to $S$. We can pull back the equation in \cite[Proposition 2.3]{BK06} along the inclusion $\sigma \times \sigma$ to get
$$ (\sigma \times \sigma)^*(\id_X \times \iota)^*\cP = (\sigma \times \sigma)^*\cP^\vee \otimes (\sigma \times \sigma)^*\pi^* \cM. $$
We need to compute the two factors in this equation.

By the above discussion we see that
$$(\sigma \times \sigma)^*\pi^* \cM = h^* \cM.$$
By \cite[(4)]{BK06} and the proof of \cite[Proposition 2.3]{BK06} we can see that
$$\iota \circ \sigma = \sigma,$$
therefore we have
$$ (\sigma \times \sigma)^*(\id_X \times \iota)^*\cP = ((\sigma \circ \id_X) \times (\iota \circ \sigma))^*\cP = (\sigma \times \sigma)^*\cP. $$
We observe that
\begin{align*}
(\sigma \times \sigma)^*\cI_{\Delta} &= (\id_\Theta \circ \sigma)^*(\sigma \circ \id_X)^*\cI_{\Delta} \\
&= (\id_\Theta \circ \sigma)^* \cI_{\Theta \times_S X \subset \XX} \\
&= \sigma^* \cI_{\Theta \subset X} \\
&= \sigma^* \cO_X(-\Theta)
\end{align*}
where the second equality uses the flatness of $\cI_\Delta$ with respect to the first projection $\pi_1$ proved in \cite[Proposition 6.15]{FMNT}. Therefore by the definition \eqref{eqn:FM-kernel} we  see that
\begin{align*}
(\sigma \times \sigma)^*\cP &= (\sigma \times \sigma)^*\cI_{\Delta} \ \otimes\ (\sigma \times \sigma)^*\pi_1^*\cO_X(\Theta) \ \otimes\ (\sigma \times \sigma)^*\pi_2^*\cO_X(\Theta) \ \otimes \ (\sigma \times \sigma)^*\pi^*\omega^{-1} \\
&= \sigma^* \cO_X(-\Theta) \otimes \sigma^* \cO_X(\Theta) \otimes \sigma^* \cO_X(\Theta) \otimes h^*\omega^{-1} \\
&= \sigma^* \cO_X(\Theta) \otimes h^*\omega^{-1}
\end{align*}
which is precisely the normal bundle of the section $\Theta$ in $X$. However the conormal bundle of $\Theta$ in $X$ is given by
$$ \sigma^* \omega_{X/S} = \sigma^* p^* \omega^{-1} = h^* \omega^{-1}, $$
where the first equality follows from \cite[(6.5)]{FMNT}. Therefore 
$$ (\sigma \times \sigma)^*\cP = h^*\omega \otimes h^*\omega^{-1} = \cO_{\Theta}. $$
Moreover  $\Theta \times_S \Theta$ lies in the smooth locus of $\XX$ on which $\cP$ is locally free. Hence 
$$ (\sigma \times \sigma)^*\cP^\vee = \cO_{\Theta}. $$
In summary, we have
$$ \cO_{\Theta} = \cO_{\Theta} \otimes h^*\cM. $$
Since $h: \Theta \to S$ is an isomorphism, we conclude 
$$ \cM = \cO_S. $$
Together with the computation of $\cA$ we get
$$ L = \cM^\vee \otimes \cA = \omega^{-1}. $$

The precise formula for $L$ depends on the choice of the Fourier-Mukai kernel $\cP$. If we take the Fourier-Mukai kernel used in \cite{BK06}, then the line bundle $L$ would be $\omega^{-3}$, which can be calculated in exactly the same way as above. However since the formula for $L$ will be irrelevant to our application, we will simply write $L$ instead of its explicit form. \qed
\end{remark}

\begin{remark}
Since the morphism $\iota$ is a morphism of $S$-schemes, it actually gives an involution on $\iota_s: X_s \to X_s$ for each fiber $X_s = p^{-1}(s)$. When $X_s$ is smooth, then $\iota_s$ is precisely the inverse operation with respect to the group law on $X_s$ with the neutral point given by $\sigma(s)$. For the description of $\iota_s$ for singular fibers, we refer to \cite[Remark 2.4]{BK06} and the reference cited there. \qed
\end{remark}

\section{Spectral sequences and applications}\label{sec:specseq}

In this section we will obtain some consequences of Theorem \ref{thm:commutativity} for the Weierstra{\ss} fibration \eqref{eqn:fibxs}. More precisely, we will apply spectral sequences on both sides of \eqref{eqn:2comp} and obtain some identities by comparing their cohomology. These identities will be used in later sections to determine the behavior of line bundles under Fourier-Mukai transforms.

For simplicity, given any $E \in \coh(X)$, the $i$-th cohomology sheaf of $\Phi(E)$ is denoted by $\Phi^iE$. We also write $n$ for $\dim X$. All functors in the spectral sequences are underived.

\subsection{The duality spectral sequences}

The composition of derived functors gives spectral sequences. The identity \eqref{eqn:2comp} in Theorem \ref{thm:commutativity} shows that two spectral sequences converge to the same limit. More precisely, for any $E \in \coh(X)$, the $E_2$-page of the spectral sequence corresponding to the composition functor $\Delta_X \circ \Phi$ is given by
$$ E_2^{p,q} = \sext^q_{\cO_X}(\Phi^{-p}E, \cO_X) \Longrightarrow E_{\infty}. $$
Notice that the non-trivial region for $(p,q)$ in the $E_2$-page is bounded by $-1 \leq p \leq 0$ and $0 \leq q \leq n$. Similarly, the $E_2$-page of the spectral sequence corresponding to the composition functor $\Phi \circ \Delta_X$ is given by
$$ E_2'^{p,q} = \Phi^q\sext^p_{\cO_X}(E, \cO_X) \Longrightarrow E'_{\infty}, $$
with the non-trivial region for $(p,q)$ in the $E_2$-page bounded by $0 \leq p \leq n$ and $0 \leq q \leq 1$. Taking the shift functor and the involution $\iota$ into consideration, we get the following result

\begin{proposition}
\label{prop:comm}
For any $E \in \coh(X)$, the following two spectral sequences converge to the same limit
$$ E_{2L}^{pq}=\sext^q_{\cO_X}(\Phi^{-p}E, \cO_X) \Longrightarrow E_{\infty} \Longleftarrow \iota^*(\Phi^{q+1}\sext^p_{\cO_X}(E, \cO_X)) \otimes p^*L=E_{2R}^{pq}. $$
Moreover, the non-trivial region for the left hand side is bounded by $-1 \leq p \leq 0$ and $0 \leq q \leq n$, while the non-trivial region for the right hand side is bounded by $0 \leq p \leq n$ and $-1 \leq q \leq 0$.
\end{proposition}

\begin{proof}
The statement follows immediately from Theorem \ref{thm:commutativity} and the above discussion. Notice that the shift in the non-trivial region for the right hand side comes from the shift functor.
\end{proof}

\begin{remark}
For better visualization, we draw the $E_2$-pages of both spectral sequences in the table form. The terms $E_{2L}^{pq}$ are given by
$$\begin{bmatrix}
\cdots & \cdots \\
\sext^3(\Phi^1E, \cO_X) & \sext^3(\Phi^0E, \cO_X) \\
\sext^2(\Phi^1E, \cO_X) & \sext^2(\Phi^0E, \cO_X) \\
\sext^1(\Phi^1E, \cO_X) & \sext^1(\Phi^0E, \cO_X) \\
\shom(\Phi^1E, \cO_X) & \shom(\Phi^0E, \cO_X)
\end{bmatrix},$$
and the terms $E_{2R}^{pq}$, up to a pullback by the involution $\iota$ and tensoring with the line bundle $p^*L$, are given by
$$\begin{bmatrix}
\Phi^1\shom(E, \cO_X) & \Phi^1\sext^1(E, \cO_X) & \Phi^1\sext^2(E, \cO_X) & \Phi^1\sext^3(E, \cO_X) & \cdots \\
\Phi^0\shom(E, \cO_X) & \Phi^0\sext^1(E, \cO_X) & \Phi^0\sext^2(E, \cO_X) & \Phi^0\sext^3(E, \cO_X) & \cdots
\end{bmatrix}.$$
Due to the limited number of rows, we immediately see that the spectral sequence on the right hand side degenerates at $E_2$-page. However the spectral sequence on the left hand side could a priori still have non-trivial arrows in $E_2$-page, but will degenerate at the latest at $E_3$-page. \qed
\end{remark}

These two spectral sequences provide us lots of information between the sheaf $E$ and its dual. More precisely, after stabilization we can compare the anti-diagonals which gives the same term in the limit. In some cases, we know the arrows among the terms $E_{2L}^{pq}$ are also trivial, then we can simply extract information from the $E_2$-pages. Some examples of such analysis are given below. 

\subsection{Applications of the duality spectral sequence}

As an application of Proposition \ref{prop:comm}, we prove the following two interesting statements about sheaves with WIT properties. We point out that in these examples both spectral sequences degenerate at $E_2$-page, because we will see that the first spectral sequence will have only one column of non-trivial terms on the $E_2$-page.

Before we state the first proposition, we collect some facts and notations. Recall that for any $\cE \in D^b(X)$, the dimensions of the supports of $\cE$ and $\Phi(\cE)$ differ at most by $1$; i.e. $\dim \cE -1 \leq \dim \Phi(\cE) \leq \dim \cE +1$. This follows from \cite[Proposition 6.1]{FMNT}.

Moreover, for any $E \in \coh(X)$, it is proven in \cite[9.2]{BMef} that $\Phi^iE=0$ unless $i=0$ or $1$. If we assume further that the support of $E$ is of codimension $c$, then we have that $\sext^i(E, \cO_X)=0$ for all $i<c$ by \cite[Proposition 1.1.6]{HL}. We define the dual sheaf of $E$ by $E^D=\sext^c(E, \cO_X)$.

We will discuss the consequences of the dualizing spectral sequence for each possibility of the difference in the dimensions assuming $E$ is a $\Phi$-WIT$_0$ or $\Phi$-WIT$_1$ sheaf on $X$. Recall that a sheaf $E \in \coh(X)$ is said to be $\Phi$-WIT$_i$ if $\Phi(E)[i] \in \coh(X)$.

\begin{proposition}
\label{prop:wit0}
Assume $E \in \coh(X)$ is $\Phi$-WIT$_0$. Then
\begin{itemize}
\item If $\dim \Phi^0E = \dim E + 1$, then $\iota^*(\Phi^0(E^D)) \otimes p^*L = (\Phi^0E)^D$;
\item If $\dim \Phi^0E = \dim E$, then $\Phi^0(E^D) = 0$; i.e., $E^D$ is $\Phi$-WIT$_1$;
\item The case $\dim \Phi^0E = \dim E - 1$ cannot happen.
\end{itemize}
\end{proposition}

\begin{proof}
We assume $\codim E = c$ for some $0 \leq c \leq n$, then we know that $\sext^q(E, \cO_X)=0$ for $0 \leq q < c$, and $E^D = \sext^c(E, \cO_X)$. This implies that $E_{2R}^{pq}=0$ for $p < c$. In other words, the possibly non-trivial terms among $E_{2R}^{pq}$ are bounded by $c \leq p \leq n$ and $-1 \leq q \leq 0$.

Since $E$ is $\Phi$-WIT$_0$, we know that $\Phi^1E=0$, hence $E_{2L}^{pq}=0$ for $p=-1$. By assumption we also have $\codim \Phi^0E \geq c-1$, hence $E_{2L}^{pq}=0$ for $p=1$ and $0 \leq q <c-1$. In other words, the only possibly non-trivial terms among $E_{2L}^{pq}$ are given by $p=0$ and $c-1 \leq q \leq n$. And it is now clear that both spectral sequences degenerate at $E_2$-page.

Now we can apply Proposition \ref{prop:comm} to compare terms in the two spectral sequences. In particular, we look at the terms in $E_2$ pages with $p+q=c-1$. By the above observations we know that $E_{2L}^{pq}=0$ for any pair of $(p,q)$ with $p+q=c-1$ unless $(p,q)=(0,c-1)$, and $E_{2R}^{pq}=0$ for any pair of $(p,q)$ with $p+q=c-1$ unless $(p,q)=(c,-1)$. Since both spectral sequences converge to the same limit by Proposition \ref{prop:comm}, we conclude that $E_{2L}^{0,c-1}=E_{2R}^{c,-1}=0$, i.e.
\begin{equation}
\label{eqn:temp1}
\sext^{c-1}(\Phi^0E, \cO_X)=\iota^*(\Phi^0\sext^c(E, \cO_X)) \otimes p^*L.
\end{equation}
When $\sext^{c-1}(\Phi^0E, \cO_X)$ is non-trivial, or equivalently $\dim \Phi^0E = \dim E + 1$, \eqref{eqn:temp1} gives precisely what we are looking for. Otherwise, we have $\dim \Phi^0E \leq \dim E$. The left hand side of \eqref{eqn:temp1} is trivial hence the right hand side is also trivial, which is equivalent to $\Phi^0\sext^c(E, \cO_X)=0$. In other words, $E^D=\sext^c(E, \cO_X)$ is a $\Phi$-WIT$_1$ sheaf.

When $\dim \Phi^0E \leq \dim E$, we show that we must have $\dim \Phi^0E = \dim E$. This can be proved by looking at the terms in the $E_2$-pages with $p+q=c$. Indeed, we get an exact sequence
$$ 0 \lra \Phi^1\sext^c(E, \cO_X) \lra \sext^c(\Phi^0E, \cO_X) \lra \Phi^0\sext^{c+1}(E, \cO_X) \lra 0. $$
Since $\sext^c(E, \cO_X) \neq 0$ and is $\Phi$-WIT$_1$, we have $\Phi^1\sext^c(E, \cO_X) \neq 0$, hence $\sext^c(\Phi^0E, \cO_X) \neq 0$, which implies that $\codim \Phi^0E \leq c$, as desired.
\end{proof}

Now we turn to sheaves with $\Phi$-WIT$_1$ property. We can similarly prove the following result.

\begin{proposition}
\label{prop:wit1}
Assume $E \in \coh(X)$ is $\Phi$-WIT$_1$. Then
\begin{itemize}
\item The case $\dim \Phi^1E = \dim E + 1$ cannot happen;
\item If $\dim \Phi^1E = \dim E$, then $\iota^*(\Phi^0(E^D)) \otimes p^*L = (\Phi^1E)^D$;
\item If $\dim \Phi^1E = \dim E -1$, then $\Phi^0(E^D) = 0$; i.e., $E^D$ is $\Phi$-WIT$_1$.
\end{itemize}
\end{proposition}

\begin{proof}
The proof is very similar to that of the previous result. We first look at the second spectral sequence. Assume that $\codim E = c$, then we know that the non-trivial region in $E_{2R}^{pq}$ is bounded by $c \leq p \leq n$ and $-1 \leq q \leq 0$. In particular, $E_{2R}^{pq}=0$ for any pair of $(p,q)$ with $p+q \leq c-2$ and the only possible non-trivial term with $p+q=c-1$ is $E_{2R}^{c,-1}=\iota^*(\Phi^0\sext^c(E, \cO_X)) \otimes p^*L$.

Then we look at the first spectral sequence. Since we assume $E$ is $\Phi$-WIT$_1$, we know that $E_{2L}^{pq}=0$ unless $p=-1$. In particular it degenerates at $E_2$-page. Moreover, by Proposition \ref{prop:comm} we can compare the two spectral sequences and conclude that $E_{2L}^{pq}=0$ for any pair $(p,q)$ with $p+q \leq c-2$, hence $\sext^q(\Phi^1E, \cO_X)=0$ for $q \leq c-1$, which implies that $\codim \Phi^1E \geq c$.

By comparing terms with $p+q=c-1$ we can also obtain that $E_{2R}^{c,-1}=E_{2L}^{-1,c}$; that is
\begin{equation}
\label{eqn:temp2}
\iota^*(\Phi^0\sext^c(E, \cO_X)) \otimes p^*L = \sext^c(\Phi^1E, \cO_X).
\end{equation}
If $\dim \Phi^1E = \dim E$, it is precisely the equation that we want. If $\dim \Phi^1E = \dim E -1$, then the right hand side of \eqref{eqn:temp2} is trivial, therefore the left hand side is also trivial, which implies that $\Phi^0\sext^c(E, \cO_X)=0$. In other words, $E^D$ is $\Phi$-WIT$_1$.
\end{proof}

Finally, we prove a very simple result along the same line, but without assuming any WIT property of $E$. We can think of it as a result which is slightly stronger than Propositions \ref{prop:wit0} and \ref{prop:wit1} in the special case of $\dim E = \dim X = n$.

\begin{proposition}
Assume $E \in \coh(X)$ satisfies $\dim E = n$ and $\dim \Phi^1E < n$, then $E^D = \shom(E, \cO_X)$ is $\Phi$-WIT$_1$.
\end{proposition}

\begin{proof}
This is a simple consequence of Proposition \ref{prop:comm}. By looking at the terms in the $E_2$-pages with $p+q=-1$, we get $\shom(\Phi^1E, \cO_X)=\iota^*(\Phi^0\shom(E, \cO_X)) \otimes p^*L$. Since we also assume $\dim \Phi^1E < n$, the left hand side of the equation is trivial, hence so is the right hand side, which implies $E^D=\shom(E, \cO_X)$ is $\Phi$-WIT$_1$.
\end{proof}

The above results will be used to study the behavior of sheaves under Fourier-Mukai transforms in the next section.

\section{Transforms of line bundles}\label{sec4}

We now use the duality spectral sequences from Section \ref{sec:specseq} to study line bundles under Fourier-Mukai transforms on elliptic fibrations.

\subsection{The basic setting}\label{subsec:setting}

In this section, we compute $\ch_0, \ch_1$ of the transforms of certain line bundles, so that we can study the slope stability of these transforms later on.

We will now make the following assumptions on our fibration $p : X \to S$:
\begin{itemize}
\item[(i)] The total space $X$ in our \wf  $p : X \to S$ is a $K$-trivial  threefold.
\item[(ii)] The canonical class $K_S$ of the base $S$ of the fibration $p$ is \textit{numerically trivial}.  This allows $S$ to be a K3 surface or an Enriques surface.
\item[(iii)] The cohomology ring of $X$ is of the form 
\[
  H^{2i}(X,\mathbb{Q}) = \Theta p^\ast H^{2i-2}(S,\mathbb{Q}) \oplus p^\ast H^{2i}(S,\mathbb{Q}),
\]
an assumption often made in the study of elliptic threefolds as in \cite[Section 6.6.3]{FMNT}.
\end{itemize}
Note that applying adjunction to the closed immersion $\Theta \hookrightarrow X$ gives
\begin{equation}\label{eqn:Sigma-sq}
\Theta^2 = \Theta\cdot p^\ast K_S.
\end{equation}

For any integer $m$, we will write
\[
  L_m = \cO_X (m\Theta).
\]
That $K_S$ is numerically trivial in assumption (ii) implies $\Theta^2=0$ and $\Theta^3=0$ by \eqref{eqn:Sigma-sq}.  Hence by the formulas in \cite[Section 6.2.6]{FMNT}, we obtain $\ch (L_m) = 1 + m\Theta$ and 
\begin{align*}
  \ch_0 (\Phi L_m) &= m, \\
  \ch_1 (\Phi L_m) &=-\Theta-\tfrac{m}{2}c
\end{align*}
where $c = -p^\ast K_S$.

We can now compute the slope of the transform of $L_m$.  Let us  fix an ample class  $\omega$  on $X$.  By assumption (iii) above, we know $\omega$ must be of the form
\begin{equation}\label{def:omega}
  \omega = t\Theta + s p^\ast H_S
\end{equation}
for some ample class $H_S$ on $S$, and some real numbers $t, s >0$. In this notation  
\[
 \omega^2 = t^2 \Theta^2 + 2ts \Theta  p^\ast H_S + s^2 p^\ast H_S^2.
\]

The slope of $\Phi L_m$ with respect to the polarisation $\omega$ is then
\begin{align}
  \mu_\omega (\Phi L_m) &= \frac{\ch_1(\Phi L_m)\omega^2}{\ch_0 (\Phi L_m)} \notag\\
  &= \frac{(-\Theta-\tfrac{m}{2}c)(t^2\Theta^2+2ts\Theta p^\ast H_S + s^2 p^\ast H_S^2)}{m} \\
  &= \frac{-s^2 \Theta p^\ast H_S^2}{m} \\
  &= -\frac{s^2}{m}H_S^2. \label{eqn:L-aslope}
\end{align}

\subsection{Transforms of reflexive sheaves}

\begin{lemma}\label{lemma-reflexivetrans}

For any $\Phi$-WIT$_0$ reflexive sheaf $E$ on $X$, the transform $\wh{E}$ is locally free.
\end{lemma}

\begin{proof}
Our definition of an elliptic fibration $p$ implies that it is Gorenstein by \cite[Definition 6.8]{FMNT}, and hence Cohen-Macaulay. 
Moreover the dualising sheaf of each fiber of $p$ is trivial.

For any $E$ satisfying the hypotheses stated, we already know $\wh{E}$ is torsion-free and reflexive by \cite[Corollary 4.5]{Lo11}.  In particular, $\wh{E}$ has homological dimension at most 1.  To show that $\wh{E}$ is locally free, we need to show that its homological dimension is zero, i.e.\ $\Hom ( (\wh{E})^\vee, \cO_x[-1])=0$ for all closed points $x \in X$.

For any closed point $x \in X$, we have
\begin{align*}
  \Hom ( (\wh{E})^\vee, \cO_x [-1]) &\cong \Hom (\cO_x^\vee [1], \wh{E}) \\
  &\cong \Hom (\cO_x [-2],\wh{E}) \\
  &\cong \Ext^2 (\cO_x, \wh{E}). \\
\end{align*}
Using the Fourier-Mukai functor $\Phi$, we also have
\[
  \Ext^2 (\cO_x,\wh{E}) \cong \Ext^2 (\wh{\cO_x}, E[-1]) \cong \Hom (\wh{\cO_x},E[1]),
\]
which vanishes for torsion-free reflexive $E$ by \cite[Lemma 4.20]{CL}.
\end{proof}

\begin{lemma}\label{lem:LB0}
For any  integer $m$, the line bundle $\cO_X (m\Theta)$ is $\Phi$-WIT$_0$ (resp.\ $\Phi$-WIT$_1$) if $m > 0$ (resp.\ $m \leq 0$).  The transform $\wh{\cO_X (m\Theta)}$ is a locally free sheaf for any $m\neq 0$, while $\wh{\cO_X} \cong \cO_\Theta \otimes p^\ast \omega$.
\end{lemma}

As above, we write $L_m = \cO_X (m\Theta)$.  

\begin{proof}
We deal with the cases of $m>0, m<0$ and $m=0$ separately.

\textbf{Case 1:} $m>0$. The restriction ${L_m}|_s$ to the fiber of $p$ over $s \in S$ is a line bundle of strictly positive degree for any $s \in S$, and so  ${L_m}|_s$ is $\Phi_s$-WIT$_0$ for any $s \in S$ \cite[Proposition 6.38]{FMNT}.  Thus $L_m$ itself is $\Phi$-WIT$_0$ \cite[Lemma 3.6]{Lo11}.  By Lemma \ref{lemma-reflexivetrans}, the transform $\wh{L_m}$ is locally free.

\textbf{Case 2:} $m<0$.  In this case, the restriction $\Lm|_s$ to the fiber of $p$ over $s \in S$ is a line bundle of  negative degree for any $s \in S$, and so $\Lm|_s$ is $\Phi_s$-WIT$_1$ for any $s \in S$ \cite[Proposition 6.38]{FMNT}.  Since $\Lm$ is torsion-free, from \cite[Lemma 3.18(ii)]{Lo11} we know that $\Lm$ is $\Phi$-WIT$_1$. 

For any positive integer $n$,  the composition of surjective sheaf morphisms $\cO_X \twoheadrightarrow \cO_{(n+1)\Theta} \twoheadrightarrow \cO_{n\Theta}$ gives the short exact sequence of sheaves
\[
 0 \to \cO_X (-(n+1)\Theta) \to \cO_X (-n\Theta) \to \cO_\Theta \to 0.
\]
Applying $\Phi$ gives the the short exact sequence of sheaves
\[
  0 \to \wh{\cO_\Theta} \to \wh{\cO_X (- (n+1)\Theta)} \to \wh{\cO_X (-n\Theta)} \to 0.
\]
Now, $\wh{\cO_\Theta}$ is flat over $S$ by \cite[Corollary 6.2]{FMNT}, and so for any point $s \in S$, we have $\wh{\cO_\Theta} |_s \cong \wh{ \cO_\Theta|_s} = \wh{\cO_z}$ by \cite[(6.2)]{FMNT}.  Since $z$ is a smooth point on $X_s$, the transform $\wh{\cO_z}$ is a line bundle.  Thus $\wh{\cO_\Theta}$  is locally free.  As a result,  to show that $\Lm = \cO_X (m\Theta)$ is locally free for all $m<0$, it suffices to check the case where $m=-1$.  That is, it suffices to show that $\wh{\cO(-\Theta)}$ is locally free.  In this case, we have the short exact sequence
\[
 0 \to \wh{\cO_\Theta} \to \wh{\cO_X (-\Theta)} \to \wh{\cO_X} \to 0.
\]
Since both $\wh{\cO_\Theta}, \wh{\cO_X}$ are flat over $S$, the transform $\wh{\cO_X (-\Theta)}$ must also be flat over $S$.  Thus for any $s \in S$ we have the short exact sequence of sheaves on $X_s$
\[
0 \to \wh{\cO_\Theta}|_s \to \wh{\cO_X (-\Theta)}|_s \to \wh{\cO_X}|_s \to 0.
\]
By Case 3 below, we know $\wh{\cO_X}|_s = \cO_z$ is the structure sheaf of a smooth point $z$ on $X_s$ ($z$ being  the intersection of $\Theta$ and $X_s$). On the other hand, we know $\wh{\cO_\Theta}|_s$ is a line bundle on $X_s$ from above.  By \cite[Corollary 6.3]{FMNT}, the restriction of the transform $\wh{\cO_X(-\Theta)}|_s$ is isomorphic to the transform of the restriction $\cO_X(-\Theta)|_s$, which is semistable hence in particular torsion free by \cite[Proposition 6.38]{FMNT}. Therefore $\wh{\cO_X(-\Theta)}|_s$ is a line bundle for any $s \in S$. It follows that $\wh{\cO_X(-\Theta)}$ has constant fiber dimension $1$ hence is a line bundle by \cite[Exercise II.5.8(c)]{Har77}.

\textbf{Case 3:} $m=0$.  That $\wh{\cO_X} \cong \cO_\Theta \otimes p^\ast \omega$ follows from \cite[Example 6.24]{FMNT}.
\end{proof}

\subsection{\textit{K}-trivial elliptic threefolds over numerically \textit{K}-trivial surfaces}

\begin{proposition}\label{prop:LB1}
Suppose $K_S$ is numerically trivial. For any ample class $\omega$ on $X$ and  any  negative integer $m$, the transform $\wh{\cO_X(m\Theta)}$ is $\mu_\omega$-stable.
\end{proposition}

\begin{proof}
Let $\omega$ be as in \eqref{def:omega}. As above, let us write $L_m$ to denote $\cO_X(m\Theta)$ for any  integer $q$.  Suppose $m$ is a  negative integer. From Lemma \ref{lem:LB0}, we know $\wh{L_{m}}$ is locally free.  Therefore, to prove that $\wh{L_{m}}$ is $\mu_\omega$-stable, it suffices to consider nonzero subsheaves $F \subsetneq \wh{L_{m}}$ where $0 < \rank F < \rank \wh{L_{m}}=|m|$ and check that $\mu_\omega (F) < \mu_\omega (\wh{L_{m}})$.

Under the assumption that $K_S$ is numerically trivial, we have from \eqref{eqn:L-aslope}
\begin{equation}\label{eqn:slopecomp0}
  \mu_\omega (\wh{L_{m}}) = \frac{s^2}{|m|} H_S^2, \end{equation}
which is strictly positive since $s>0$.

Now, let us write $n=-m$ and consider the structure short exact sequence in $\Coh (X)$
\begin{equation}\label{sd-eq4}
  0\to \cO_X (-n\Theta) \to \cO_X \to \cO_{n\Theta} \to 0
\end{equation}
which is taken by $\Phi$ to the following short exact sequence of sheaves
\begin{equation}\label{sd-eq1}
 0 \to \wh{\cO_{n\Theta}} \to \wh{\cO_X (-n\Theta)} \to \cO_{\Theta}\otimes p^\ast \omega \to 0.
\end{equation}
Writing $F'$ to denote the image of $F$ under the surjection $\wh{\cO_X(-n\Theta)} \twoheadrightarrow \cO_\Theta \otimes p^\ast \omega$ in \eqref{sd-eq1} and writing $F'' := \wh{\cO_{n\Theta}} \cap F$, we have a commutative diagram where both rows are short exact sequences and the vertical maps are all inclusions of sheaves:

\begin{equation}\label{eq60}
\xymatrix{
   0 \ar[r] & F'' \ar[r] \ar[d] & F \ar[r] \ar[d] & F' \ar[r] \ar[d] & 0 \\
  0 \ar[r] &  \wh{\cO_{n\Theta}} \ar[r] & \wh{\cO_X (-n\Theta)} \ar[r] & \cO_{\Theta} \otimes p^\ast \omega \ar[r] & 0 
}.
\end{equation}

Thus 
\begin{align}
  \mu_\omega (F)  &= \frac{\omega^2 \ch_1(F)}{\ch_0 (F)} \notag \\
  &= \frac{\ch_1(F)(t^2\Theta^2 + 2ts\Theta p^\ast H_S + s^2 p^\ast H_S^2)}{\ch_0 (F)} \notag \\
  &= \frac{\ch_1(F) (t^2\Theta^2 + 2ts\Theta p^\ast H_S)}{\ch_0 (F)} + \frac{\ch_1(F)\cdot s^2 p^\ast H_S^2}{\ch_0 (F)} \notag \\
  &= \frac{(\ch_1(F'')+\ch_1(F'))(t^2\Theta^2 + 2ts\Theta p^\ast H_S)}{\ch_0 (F)} + \frac{\ch_1(F) \cdot s^2 p^\ast H_S^2}{\ch_0 (F)}. \label{eqn:Fslope}
\end{align}
Since $L_{-n}$ is a line bundle, its restriction to the generic fiber of $p$ is a stable torsion-free sheaf; hence the restriction of $\wh{L_{-n}}$ to the generic fibr of $p$ is also a stable torsion-free sheaf.  Since $0<\rank F < \rank \wh{L_{-n}}$, we have 
\[
   \frac{\ch_1(F).f}{\ch_0 (F)} < \frac{\ch_1(\wh{L_{-n}})\cdot f}{\ch_0 (\wh{L_{-n}})} = \frac{1}{n}.
\]
Since $\ch_1(F)\cdot f$ is an integer, we must have $\ch_1(F)\cdot f \leq 0$.  Hence 
\begin{equation}\label{eqn:slopecomp1}
  \frac{\ch_1(F)\cdot s^2 p^\ast H_S^2}{\ch_0 (F)} \leq 0 < \frac{s^2}{n}  H_S^2.
\end{equation}

Now, that $\wh{\cO_{n \Theta}}$ has a filtration by line bundles in $\Coh (X)$ implies that, in $\Coh (X)$,  $F''$ has a filtration by twists by line bundles of ideal sheaves  of  proper closed subscheme of $X$.  Hence $-\ch_1(F'')$ is an effective divisor on $X$. 

We claim that for any effective divisor $D$ in $X$, we have
\begin{equation*}
  D \cdot (t^2\Theta^2 + 2ts\Theta \cdot p^\ast H_S) \geq 0.
\end{equation*}
Indeed, since $\Theta^2=0$, it suffices to show $D \cdot \Theta \cdot p^*H_S \geq 0$. Without loss of generality, we can assume that $D$ is irreducible. If $D=\Theta$, then $D \cdot \Theta \cdot p^*H_S = 0$; so let us  assume that $D$ meets $\Theta$ along some curve $C \subseteq \Theta$. Since $H_S$ is ample, $p^*H_S$ is also ample on $\Theta$, hence  $D \cdot \Theta \cdot p^*H_S = C \cdot p^*H_S \geq 0$.

By the conclusion of the previous paragraph, we now have
\begin{equation*}
  \ch_1(F'')(t^2\Theta^2 + 2ts\Theta \cdot p^\ast H_S) \leq 0.
\end{equation*}

As for $\ch_1(F')$, since $F'$ is a subsheaf of $\cO_\Theta \otimes p^\ast \omega$ and $\Theta$ is irreducible, we know $\ch_1(F')$  is either zero or $\Theta$.  Thus $\ch_1(F')(t^2\Theta^2 + 2ts\Theta \cdot p^\ast H_S) = 0$.

Combining the various inequalities in the last three paragraphs, we now have
\[
  \mu_\omega (F) \leq 0 < \mu_\omega (\wh{L_{-n}}),
\]
showing that $\wh{L_{-n}}$ is $\mu_\omega$-stable.
\end{proof}

\begin{theorem}\label{theorem:main1}
Suppose $p : X \to S$ is a  Weierstra{\ss}  threefold where $X$ is $K$-trivial and  $K_S$ is numerically trivial. Suppose $\omega$ is an ample class on $X$.    Then for any line bundle $M$ on $X$ of nonzero fiber degree, the transform $\wh{M}$ is a $\mu_\omega$-stable locally free sheaf.
\end{theorem}

\begin{proof}
By assumption (iii) in the beginning of Section \ref{subsec:setting},  any line bundle on $X$ is of the form $\cO_X (m\Theta)$ for some $m \in \mathbb{Z}$ up to tensoring by a line bundle pulled back from the base.  By projection formula, tensoring with a sheaf pulled back from the base $S$ commutes with the Fourier-Mukai functor $\Phi$.  As a result, it suffices to prove the theorem for line bundles of the form $\cO_X (m\Theta)$ where $m$ is a nonzero integer.  The case where $m<0$  is proved in Proposition \ref{prop:LB1}.  The case where $m>0$ follows from the case of $m<0$ together with Proposition \ref{prop:wit1}, and the fact that slope stability is preserved under taking the dual.
\end{proof}

\bibliographystyle{plain}
\bibliography{refsLZ3}{}

\end{document}